\documentclass[a4paper]{article}
\usepackage{latexsym}
\usepackage{sss}
\usepackage{graphicx} 
\usepackage{epsfig} 
\usepackage{amsbsy,amsfonts,amsmath,amssymb,amsthm,ascmac,mathrsfs,centernot}
\usepackage{flushend} 
\usepackage{hyperref}
\usepackage{subcaption}
\usepackage{multirow}

\theoremstyle{plain}
\numberwithin{equation}{section}
\newtheorem{thm}{Theorem}[section]

\newcommand{\lbeq}[1]{\label{eq:#1}}

\newcommand{\N}{{\mathbb N}}
\newcommand{\nn}{\nonumber}

\newcommand{\PSCA}{P^{\sss\rm SCA}}
\newcommand{\piG}{\pi^{\sss\rm G}}

\newcommand{\refeq}[1]{(\ref{eq:#1})}
\newcommand{\rhoTM}{\rho_{\sss\rm TM}}
\newcommand{\sss}{\scriptscriptstyle}

\newcommand{\vep}{\varepsilon}

\newcommand{\Evec}  {\textbf{E}}

\newcommand{\Pvec}  {\textbf{P}}

\newcommand{\qvec}  {\boldsymbol{q}}

\newcommand{\sigmavec}  {\boldsymbol{\sigma}}
\newcommand{\tauvec}  {\boldsymbol{\tau}}

\begin{document}
\date{}
\title{\LARGE \textbf{Stochastic optimization: Glauber dynamics versus stochastic cellular automata}
}
\author{
Bruno Hideki Fukushima-Kimura${}^{1}$,
Yoshinori Kamijima${}^{2}$,
Kazushi Kawamura${}^{3}$ and
Akira Sakai${}^{1}$\\[1ex]
\small ${}^1$Faculty of Science, Hokkaido University \\ 
\small Kita 10, Nishi 8, Kita-ku, Sapporo, Hokkaido 060-0810, Japan\\
\small ${}^2$ National Center for Theoretical Sciences (NCTS) \\
\small No. 1, Sec. 4, Roosevelt Rd., Taipei City 106, Taiwan\\
\small ${}^3$Institute of Innovative Research, Tokyo Institute of Technology \\
\small Nagatsuta-cho 4259-J3-30, Midori-ku, Yokohama, Kanagawa 226-8503, Japan\\
\smallskip
E-mail: \texttt{bruno@math.sci.hokudai.ac.jp, kamijima@ncts.ntu.edu.tw,}\\
\texttt{kawamura@artic.iir.titech.ac.jp, sakai@math.sci.hokudai.ac.jp}
}

\maketitle
\thispagestyle{empty}
\abstract{
The topic we address in this paper concerns the minimization of a Hamiltonian function for an Ising model through the application of simulated annealing algorithms based on (single-site) Glauber dynamics and stochastic cellular automata (SCA).  Some rigorous results are presented in order to justify the application of simulated annealing for a particular kind of SCA. After that, we compare the SCA algorithm and its variation, namely the $\vep$-SCA algorithm, studied in this paper with the Glauber dynamics by analyzing their accuracy in obtaining optimal solutions for the max-cut problem on Erd\" os-R\'enyi random graphs, the traveling salesman problem (TSP), and the minimization of Gaussian and Bernoulli spin glass Hamiltonians. We observed that the SCA performed better than the Glauber dynamics in some special cases, while the $\vep$-SCA showed the highest performance in all scenarios. 
}

\section{Introduction}\label{sec:intro}
In several problems involving complex networks, there has been a growing demand to provide algorithms that can offer optimal solutions for large-scale combinatorial problems within a relatively short amount of time. Due to the fact that no polynomial-time algorithms for NP-hard problems are known \cite{gj79}, an alternative approach is necessary.  It follows from the fact that a considerable amount of problems of practical interest can be mapped into the problem of finding a ground state of an Ising model (e.g., \cite{BGJR88, c85, l14, Drug}), there are several algorithms being developed recently aiming at approaching such ground states, see \cite{A2019, GAT21, GTD19, osky19, statica}. 

The robustness of the statistical mechanical framework has received a lot of attention in recent decades, and a stochastic approach to combinatorial problems has been widely employed. The approach adopted in this paper\footnote{The material of this paper was partially presented at the 53rd Annual Conference of the Institute of Systems, Control and Information Engineers (SCI’21) which was held in October, 2021} continues the one discussed in \cite{SSS21}, and we extend some of its theoretical results and provide further examples. As in the previous work, our main problem of interest is the problem of determining one of the ground states (i.e., the points of global minima) of a Hamiltonian function $H$ for an Ising model on a finite simple graph $G=(V,E)$. 
Given a collection of spin-spin coupling constants  $(J_{x,y})_{x,y\in V}$ such that $J_{x,y} = J_{y,x}$, and
$J_{x,y}=0$ if $\{x,y\}\notin E$, and local external fields 
$(h_x)_{x\in V}$, let us consider the Ising Hamiltonian defined by
\begin{align}\label{eq:Hamiltonian}
H(\sigmavec)=-\frac12\sum_{x,y\in V}J_{x,y}\sigma_x\sigma_y-\sum_{x\in V}h_x
 \sigma_x
\end{align}
for each spin configuration $\sigmavec = (\sigma_x)_{x\in V}\in\{-1,+1\}^V$. 

Approximating an Ising Hamiltonian's ground states can be achieved by applying simulated annealing based on a Markov chain Monte Carlo spin-flip dynamics. The most usual example consists of considering a time-inhomogeneous Markov chain based on a single spin-flip dynamics (such as the Glauber or Metropolis dynamics \cite{g63}) while slowly decreasing its temperature towards zero. Some theoretical results (e.g., \cite{Cat92, h88}) show that if the temperature is decreased in time $t$ at a certain speed rate, such a procedure yields an approximation to the minimum value of $H$. In \cite{osky19, statica}, the authors considered simulated annealing based on parallel spin-flip dynamics that, differently from the usual method, allows the system to update multiple spins independently, and simulations showed that it manages to converge to the ground states significantly faster compared to the single spin-flip methods.
In addition to these observations, in \cite{SCA21} the authors investigated from the mathematical point of view a specific kind of stochastic cellular automata (SCA) which was derived from \cite{dss12,st18}, in order to determine sufficient conditions under which the convergence to the ground states is guaranteed. 

As introduced in \cite{SCA21, hkks19, dss12, st18}, the SCA transition kernel at inverse temperature $\beta \geq 0$ for a collection of (non-negative) pinning parameters $\qvec = (q_x)_{x\in V}$ is given by
\begin{align}\lbeq{PSCAdef}
\PSCA_{\beta,\qvec}(\sigmavec,\tauvec)=\prod_{x\in V}\frac{e^{\frac\beta2(\tilde h_x(\sigmavec)+q_x\sigma_x)
 \tau_x}}{2\cosh(\frac\beta2(\tilde h_x(\sigmavec)+q_x\sigma_x))}
\end{align}
for any  spin configurations $\sigmavec,\tauvec$  in $\{-1, +1\}^V$, where the cavity fields $\tilde h_x$ are defined by letting $\tilde h_x(\sigmavec)=\sum_{y\in V}J_{x,y}\sigma_y+h_x$. We immediately notice that the right-hand side of the equation above is expressed in terms of a product of local probabilities $p_{x,q}(\cdot |\sigmavec)$ given by
\begin{equation}
p_{x,q}(s |\sigmavec) = \frac{e^{\frac\beta2(\tilde h_x(\sigmavec)+q\sigma_x)s}}{2\cosh(\frac\beta2(\tilde h_x(\sigmavec)+q\sigma_x))} 
\end{equation} 
for $s \in \{-1,+1\}$. Therefore, whenever we update the state of the system, all spins 
are updated simultaneously and independently according to a probabilistic rule for each vertex that depends on the current spin configuration. Such characteristics imply that the system can potentially
transition from any spin configuration to another in a single step, contrasting with the standard Glauber and Metropolis dynamics.
Note that due to the presence of the pinning parameters in the local transition probabilities, the SCA tends to flip a smaller amount of spins from certain configurations, especially when such parameters are taken sufficiently large and the temperature approaches zero. It follows that, for SCA-based simulated annealing algorithms, such an effect typically slows down the dynamics, which may prevent the system from reaching lower energetic configurations, however, as we observe in this paper (see also \cite{SSS21}), such algorithms still performs better than Glauber dynamics with regards to the success rate of reaching a ground state in certain scenarios.

In addition to the SCA defined in \refeq{PSCAdef}, let us also consider the so-called $\vep$-SCA. First introduced in \cite{SCA21, SSS21}, the $\vep$-SCA is also a particular kind of probabilistic cellular automata, but with the advantage of having attenuated temperature-dependent pinning effects, which encourages a larger number of spin-flips at low temperatures compared to the SCA.
Given the inverse temperature $\beta \geq 0$ and a parameter $\vep \in (0,1]$, the 
transition kernel of the $\vep$-SCA is given by
\begin{equation}
P_{\beta,\vep}(\sigmavec,\tauvec) = \prod_{x: \,\tau_x = - \sigma_x }\big(\vep p_x(\sigmavec)\big)\prod_{y: \, \tau_y = \sigma_y }
 \big(1-\vep p_y(\sigmavec)\big)
\end{equation}
for every pair $\sigmavec,\tauvec$ of configurations in $\{\-1,+1\}^V$, where $p_x(\sigmavec)$ is defined by
\begin{equation}\lbeq{flipprobability}
p_x(\sigmavec) = \frac{e^{-\frac\beta2\tilde h_x(\sigmavec) \sigma_x}}
 {2\cosh(\frac\beta2 \tilde h_x(\sigmavec))}.
\end{equation} 
Note that $p_x(\sigmavec)$ coincides with $p_{x,0} (-\sigma_x| \sigmavec)$, so this algorithm can be interpreted in the following way: the spins eligible to be flipped are randomly assigned with probability $\vep$ independently of each other, then, only those spins which were assigned will be updated simultaneously and independently, where the probability of flipping the spin $\sigma_x$ is $p_x(\sigmavec)$, while the unassigned spins remain unchanged.  Note that $P_{\beta, \vep}$ can also be written as a product of local transition probabilities just like in equation \refeq{PSCAdef}, where each of the corresponding local probabilities  $\tilde{p}_{x,\vep}(\cdot |\sigmavec)$ is given by
\begin{equation}\label{eq:elocal}
\tilde{p}_{x, \vep}(s |\sigmavec) = (1 - \vep) \delta_{\sigma_x, s} + \vep p_{x,0}(s | \sigmavec)
\end{equation} 
for each $s \in \{-1, +1\}$.

Differently from the SCA and some of the well-established single spin-flip dynamics, there is still no solid mathematical theory that describes in which circumstances a simulated annealing algorithm based on the $\vep$-SCA can be successfully applied such as in Theorem \ref{thm:SAforSCA}. Even though, in recent investigations, e.g., \cite{SCA21, SSS21}, the $\vep$-SCA outperformed the tested algorithms in all scenarios, managing to reach lower energy configurations and the ground states at a higher success rate, provided the parameter $\vep$ was chosen appropriately. In this paper, we provide a natural extension for the preliminary results from \cite{SSS21}, where we derive a new result concerning an upper bound for the mixing time of the $\vep$-SCA,  and test the algorithms by applying them to a larger class of problems, including the traveling salesman problem (TSP).

\section{Mathematical results}\label{sec:MT}
In this section, we summarize some of the rigorous results which are known for the SCA and the $\vep$-SCA. Furthermore, we also provide an upper bound for the mixing time of the $\vep$-SCA. For the readers interested in the proofs concerning to the SCA,  refer to \cite{SCA21}. The first mathematical results, Theorems \ref{thm:mixing} and \ref{thm:emixing},  show that the mixing time $t_{mix}(\delta)$ (see definition in \cite{lp17}), that is, the time the system takes to be close to equilibrium, is bounded above by a quantity which is proportional to $\log|V|$ as long as the temperature is taken sufficiently high. Such results show that both the SCA and the $\vep$-SCA reach the equilibrium faster than Glauber dynamics which is known to be at least proportional to $|V| \log |V|$, see \cite{DP11, HS07}. 

\begin{thm}[Mixing time \cite{SCA21}]\label{thm:mixing}
For any non-negative pinning parameters $\qvec = (q_x)_{x \in V}$, if $\beta$ is sufficiently small so that
\begin{align}\lbeq{mixingcond}
r\equiv\max_{x\in V}\bigg(\tanh\frac{\beta q_x}2+\sum_{y\in V}\tanh\frac{\beta
 |J_{x,y}|}2\bigg)
 <1,
\end{align}
then, the mixing time $t_{\rm mix}(\delta)$ satisfies
\begin{equation}
t_{\rm mix}(\delta)\le\bigg\lceil\frac{\log|V|-\log\delta}{\log(1/r)}\bigg\rceil.
\end{equation}
\end{thm}

\begin{thm}[Mixing time of the  $\vep$-SCA]\label{thm:emixing}
For any positive $\vep \in (0,1]$, if $\beta$ is sufficiently small such that
\begin{align}\lbeq{mixingcond2}
r\equiv  (1-\vep) + \vep \max_{x\in V}\bigg(\sum_{y\in V}\tanh\frac{\beta
 |J_{x,y}|}2\bigg)
 <1,
\end{align}
then, the mixing time $t_{\rm mix}(\delta)$ satisfies
\begin{equation}
t_{\rm mix}(\delta)\le\bigg\lceil\frac{\log|V|-\log\delta}{\log(1/r)}\bigg\rceil.
\end{equation}
\end{thm}

\begin{proof}
Following the same strategy as adopted in the proof of \cite[Proposition 2.2]{SCA21}, it is sufficient to prove that $\rhoTM(P_{\beta,\vep}(\sigmavec,\cdot),P_{\beta, \vep}(\tauvec,\cdot))\le r$ holds for any configurations $\sigmavec,\tauvec$ satisfying 
$|D_{\sigmavec,\tauvec}|=1$, where $\rhoTM$ stands for the transportation metric and $D_{\sigmavec,\tauvec}=\{x\in V: \sigma_x\ne\tau_x\}$.  

Let us denote the local probabilities $\tilde{p}_{y,\epsilon}(+1|\sigmavec)$ from equation (\ref{eq:elocal}) simply by $p(\sigmavec,y)$, and let us assume that $D_{\sigmavec,\tauvec}=\{x\}$, that is, $\tauvec$ is the configuration $\tauvec=\sigmavec^x$ which coincides with $\sigmavec$ apart from its value at $x$.  
Note that $p(\sigmavec,y) \neq p(\sigmavec^x,y)$ only if $y=x$ or 
$y\in N_x\equiv\{v\in V:J_{x,v}\ne0\}$.  Using this as a threshold function for 
i.i.d.~uniform random variables $(U_y)_{y\in V}$ on $[0,1]$, we define the 
coupling $(X,Y)$ of $P_{\beta,\vep}(\sigmavec,\cdot)$ and 
$P_{\beta,\vep}(\sigmavec^x,\cdot)$ as
\begin{align}\lbeq{coupling-SCA}
X_y=\begin{cases}
 +1 & [U_y\le p(\sigmavec,y)],\\
 -1 & [U_y>p(\sigmavec,y)],
\end{cases}
\end{align}
and
\begin{align}
Y_y=\begin{cases}
 +1 & [U_y\le p(\sigmavec^x,y)],\\
 -1 & [U_y>p(\sigmavec^x,y)].
\end{cases}
\end{align}
If we denote the probability measure of this coupling by $\Pvec_{\sigmavec,\sigmavec^x}$ and 
its expectation by $\Evec_{\sigmavec,\sigmavec^x}$, then, it follows that
\begin{align}
\Evec_{\sigmavec,\sigmavec^x}[|D_{X,Y}|] 
= &  |p(\sigmavec,x)-p(\sigmavec^x,x)| \nn \\
&+ \sum_{y\in N_x} |p(\sigmavec,y)-p(\sigmavec^x,y)|, 
\end{align}
where, 
\begin{align}
|p(\sigmavec,x)-p(\sigmavec^x,x)| = 1- \vep,
\end{align}
and for $y\in N_x$ we have
\begin{align}
|p(\sigmavec,y)-p&(\sigmavec^x,y)| \nn   \\
\le \frac{\vep}{2}\bigg|&\tanh\bigg(\frac{\beta(\sum_{v \ne x}J_{v,y}\sigma_v+h_y)}2+\frac{\beta J_{x,y}}2\bigg)\nn \\
&-\tanh\bigg(\frac{\beta(\sum_{v\ne x}J_{v,y}\sigma_v+h_y)}2 -\frac{\beta J_{x,y}}2\bigg)\bigg|.
\end{align}
Since $|\tanh(a+b)-\tanh(a-b)|\le2\tanh|b|$ for any $a,b$, then, we conclude that
\begin{align}
\rhoTM\Big(\PSCA_{\beta,\qvec}(\sigmavec,\cdot),\PSCA_{\beta,\qvec}(\sigmavec^x,
 \cdot)\Big)&\le\Evec_{\sigmavec,\sigmavec^x}[|D_{X,Y}|] \nn \\
 \le (1 - \vep)
 &+\vep \sum_{y\in N_x}\tanh\frac{\beta|J_{x,y}|}2\le r,
\end{align}
as required.
\end{proof}

The next result shows that if we consider the inhomogeneous Markov chain where the system is updated at time $t$ according to the SCA transition kernel at inverse temperature $\beta_t$ proportional to $\log t$, then, the algorithm will converge to the uniform distribution concentrated over the ground states of the Hamiltonian (\ref{eq:Hamiltonian}).

\begin{thm}[Simulated annealing \cite{SCA21}]\label{thm:SAforSCA}
Let us assume that each $q_x \geq \lambda/2$, where $\lambda$ is the largest eigenvalue of the matrix $[-J_{x,y}]_{x,y \in V}$. If we choose $\{\beta_t\}_{t\in\N}$ as
\begin{equation}
\beta_t=\frac{\log t}\Gamma,
\end{equation}
where
$\Gamma=\sum_{x \in V} \Gamma_x$ and 
$\Gamma_x =q_x+|h_x|+\sum_{y\in V}|J_{x,y}|$,
then,  for any initial $j\in\N$, we have
\begin{align}\label{SA:Sergodic}
\lim_{t\to\infty}\sup_\mu\big\|\mu\PSCA_{\beta_j,\qvec}\PSCA_{\beta_{j+1},
 \qvec}\cdots\PSCA_{\beta_t,\qvec}-\piG_\infty\big\|_{\sss\rm TV} = 0.
\end{align}
\end{thm} 

It is important to point out that the condition on the values of the pinning parameters $\qvec = (q_x)_{x\in V}$ to be at least $\lambda / 2$ is just a sufficient condition related to a more general one originated in  \cite{osky19}. Therefore, such bound still has some room for improvement and the algorithm may still converge even if we consider smaller values for such parameters.
The importance of letting the parameters $q_x$ assume relatively large values comes from the fact that it prevents the system from flipping many spins at the same time, which helps to guarantee the convergence of the method. Although the pinning parameters restrict the system from making certain transitions, which typically slows down the dynamics, especially at low temperatures, they are necessary to avoid a certain kind of oscillatory behavior characterized by an alternation between two distinct spin configurations as the temperature drops close to zero, see \cite[Fig. 1]{SSS21}. Such oscillatory behavior tends to manifest when the pinning parameters assume values between $0$ and $\lambda /2$. This phenomenon can be widely observed, for example, in anti-ferromagnetic Ising models, due to the fact that, in this case, each spin tends to misalign with its neighbors, in such a way that,  if the pinning parameters are not set appropriately, it may result in an oscillation where the majority of spins are simultaneously updated to $+1$, and in the next step, the majority of the spins are updated to $-1$, and so on. Similarly, for the $\vep$-SCA, if the parameter $\vep$ takes values in an interval close to $1$, more spin-flips per update are allowed, and its behavior tends to be similar to the behavior of an SCA without pinning parameters, which in some cases may also introduce the same oscillatory behavior. For that reason, more theoretical studies are necessary in order to derive an analogous assumption for the parameter $\vep$
 so that we can have a theoretical limit for its value below which the convergence is assured.

\section{Simulations}

In this section we exemplify the application of the algorithms presented in Section \ref{sec:intro} in the search for ground states corresponding to three famous problems:
\begin{enumerate}
\item \emph{Spin glasses.}  Let us consider a spin glass Hamiltonian in a complete graph with $N$ vertices, without external fields (i.e., $h_x =0$), where the values for the spin-spin couplings $J_{x,y} = J_{y,x}$  are realizations of i.i.d. random variables. Two cases are analyzed: the Gaussian case, where each $J_{x,y} $ is a realization of a  standard normal random variable, and the Bernoulli case, where $\mathbb{P}(J_{x,y} = +1) = 1 - \mathbb{P}(J_{x,y} = -1) = p$.

\item \emph{Max-cut problem.} The Hamiltonian is defined in an Erd\" os-R\' enyi random graph with $N$ vertices and edge probability $p$, without external fields, and spin-spin coupling  satisfying $J_{x,y} = -1$ if $\{x,y\}$ is an edge of the graph and $J_{x,y} =0$ otherwise.  

\item \emph{Traveling salesman problem (TSP).} The Hamiltonian considered corresponds to the problem of finding the shortest path that connects $n$ cities, visiting each city only once, and returning to the initial point, see \cite{l14} for more details.

\end{enumerate}

Despite the fact one of our theoretical results (Theorem \ref{thm:SAforSCA}) and some of the well-established theorems on simulated annealing (e.g., \cite{h88}) ensure the convergence of certain annealing algorithms to the ground states of a Hamiltonian when the temperature decreases at a logarithmic rate towards zero, practical simulations involving parallel spin updates have been successfully applied by decreasing the temperature exponentially fast, e.g., \cite{SSS21, osky19, statica}. 

In the following, we consider the problems described above in the particular case where the underlying graph contains $N=100$ vertices and compare the performances of simulated annealing algorithms based on the $\vep$-SCA, SCA (with pinning parameters $q_x = \lambda / 2$) and Glauber dynamics.
For each comparison, we ran $1000$ trials for each algorithm, where each trial consisted of taking an initial spin configuration uniformly at random and applying $10^4$ Markov chain steps with the exponential temperature cooling schedule given by
\begin{equation}
\beta_t = \beta_0 \exp{(\alpha t)},
\end{equation}
where $\beta_0 = \alpha = 10^{-3}$. After that, the histogram of the smallest energy reached within each trial was generated. In the particular case where we analyze the TSP,  for the sake of simplicity, we considered only $n = 10$ cities, that is, $N = n^2 = 100$. The distances between cities $i$ and $j$,  namely $d_{i, j} = d_{j , i}$, were randomly selected integers between $1$ and $100$, and the parameters $A$ and $B$ (see \cite{l14}) were chosen as $A = \max_{i, j} d_{i, j}$ and $B = 1$ . 

The histograms respective to each studied case are illustrated in Figs. \ref{tab:hist1} and \ref{tab:hist2}, and results of the simulations are summarized in Table \ref{tab:SR}.

\begin{figure}
        \centering   
        \begin{subfigure}[b]{0.45\textwidth}
                \includegraphics[width=\textwidth]{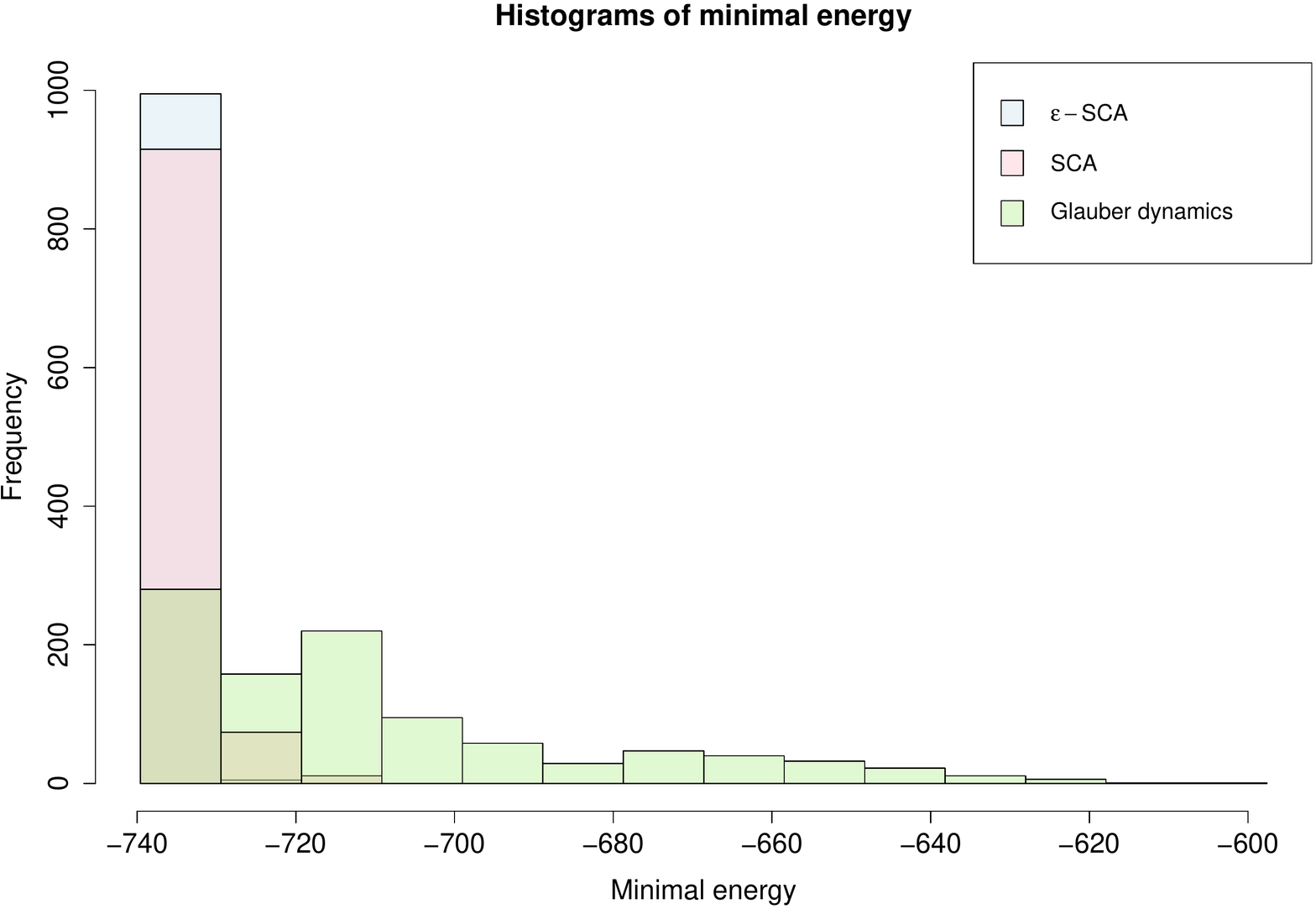}
                \caption{Gaussian case, where $\vep = 0.9$}
                \label{fig:GSG}
        \end{subfigure}
        \begin{subfigure}[b]{0.45\textwidth}
                \includegraphics[width=\textwidth]{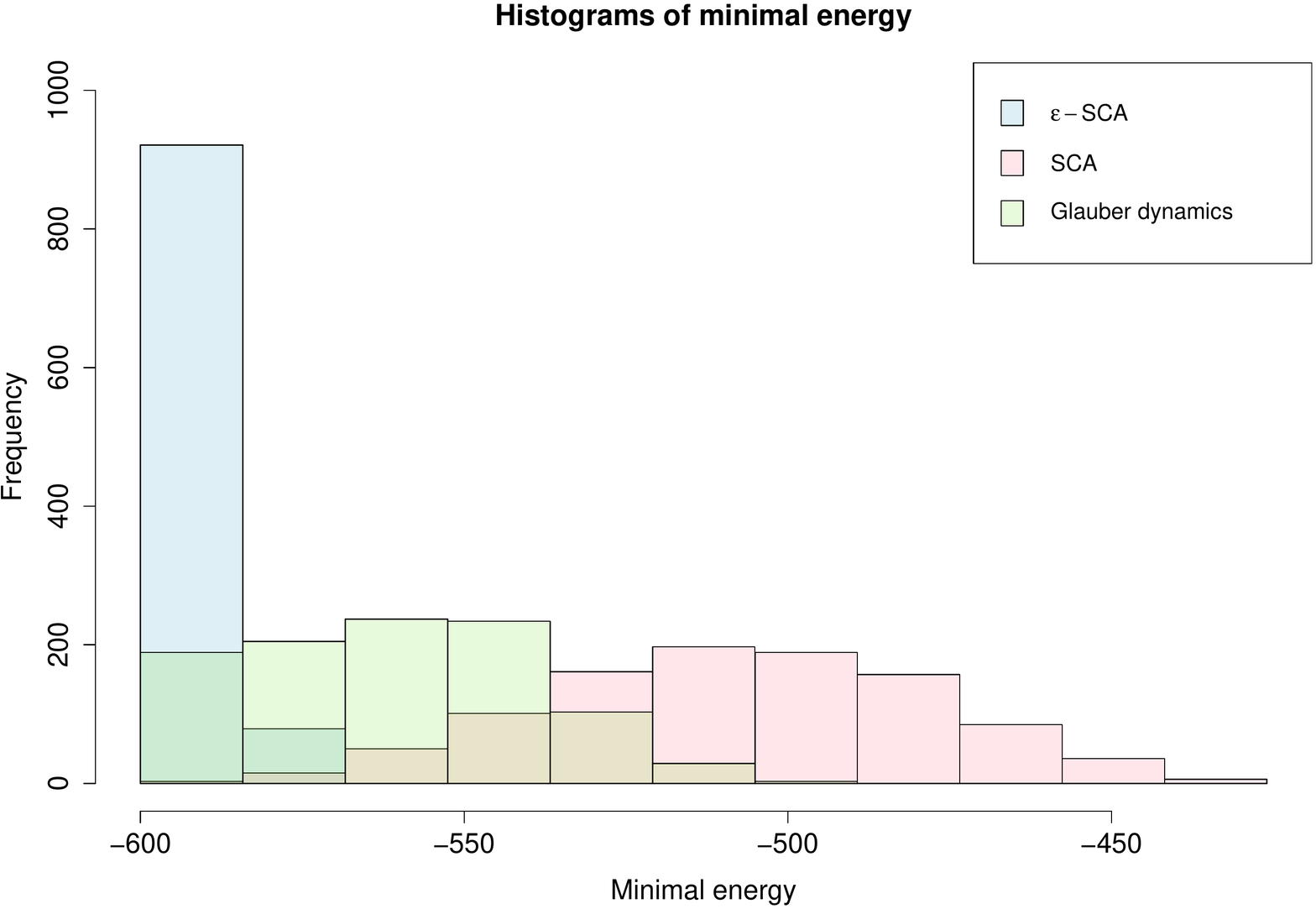}
                \caption{Bernoulli case with $p=0.2$,  where $\vep = 0.35$}
                \label{fig:BSG1}
        \end{subfigure}
        \begin{subfigure}[b]{0.45\textwidth}
                \includegraphics[width=\textwidth]{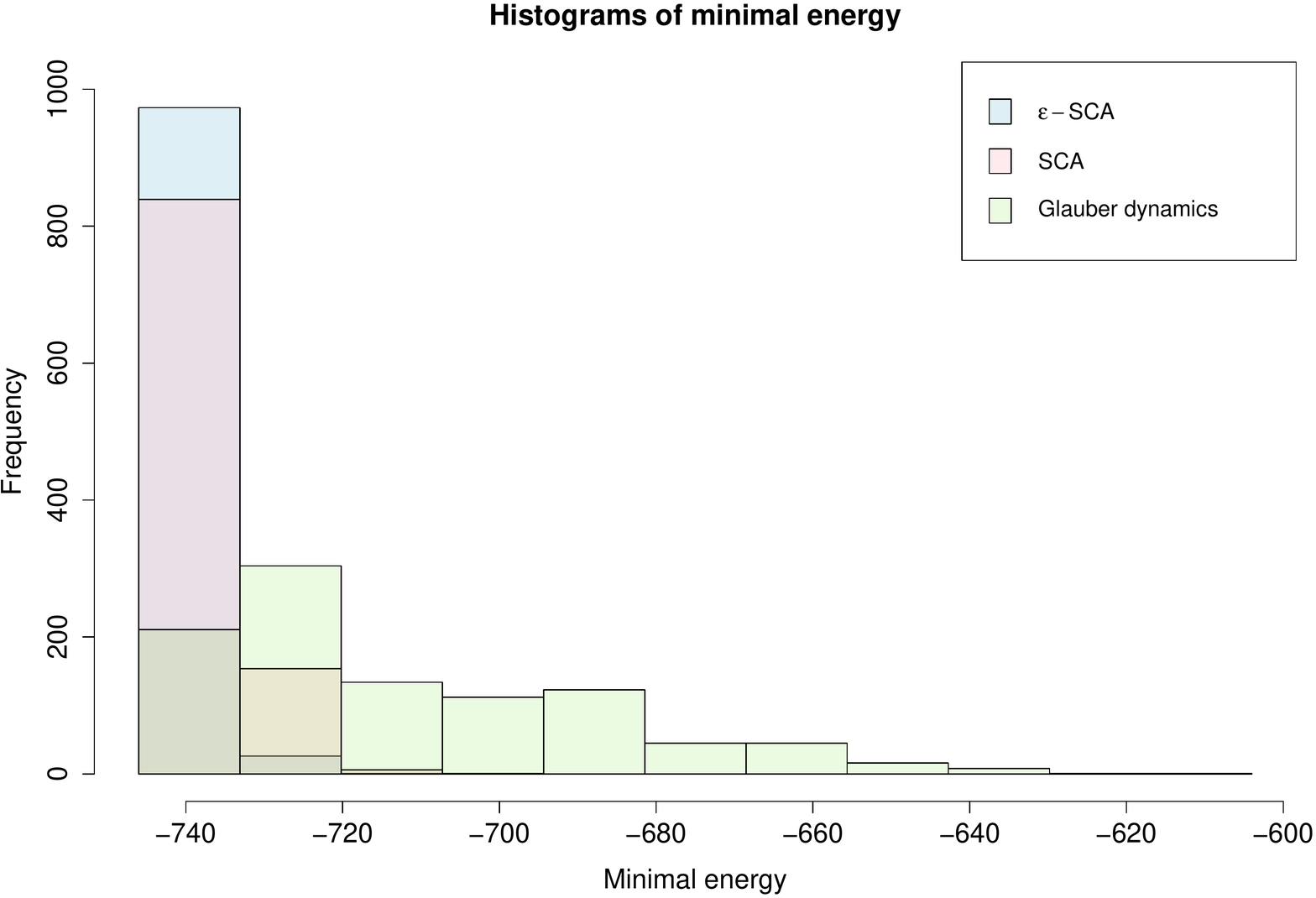}
                \caption{Bernoulli with $p=0.5$,  where $\vep = 0.75$}
                \label{fig:BSG2}
        \end{subfigure}
    \caption{Histograms obtained by using the $\vep$-SCA, SCA, and Glauber dynamics for spin glasses on the complete graph with $N = 100$.}
    \label{tab:hist1}
\end{figure}

\begin{figure}
        \centering   
        \begin{subfigure}[b]{0.45\textwidth}
                \includegraphics[width=\textwidth]{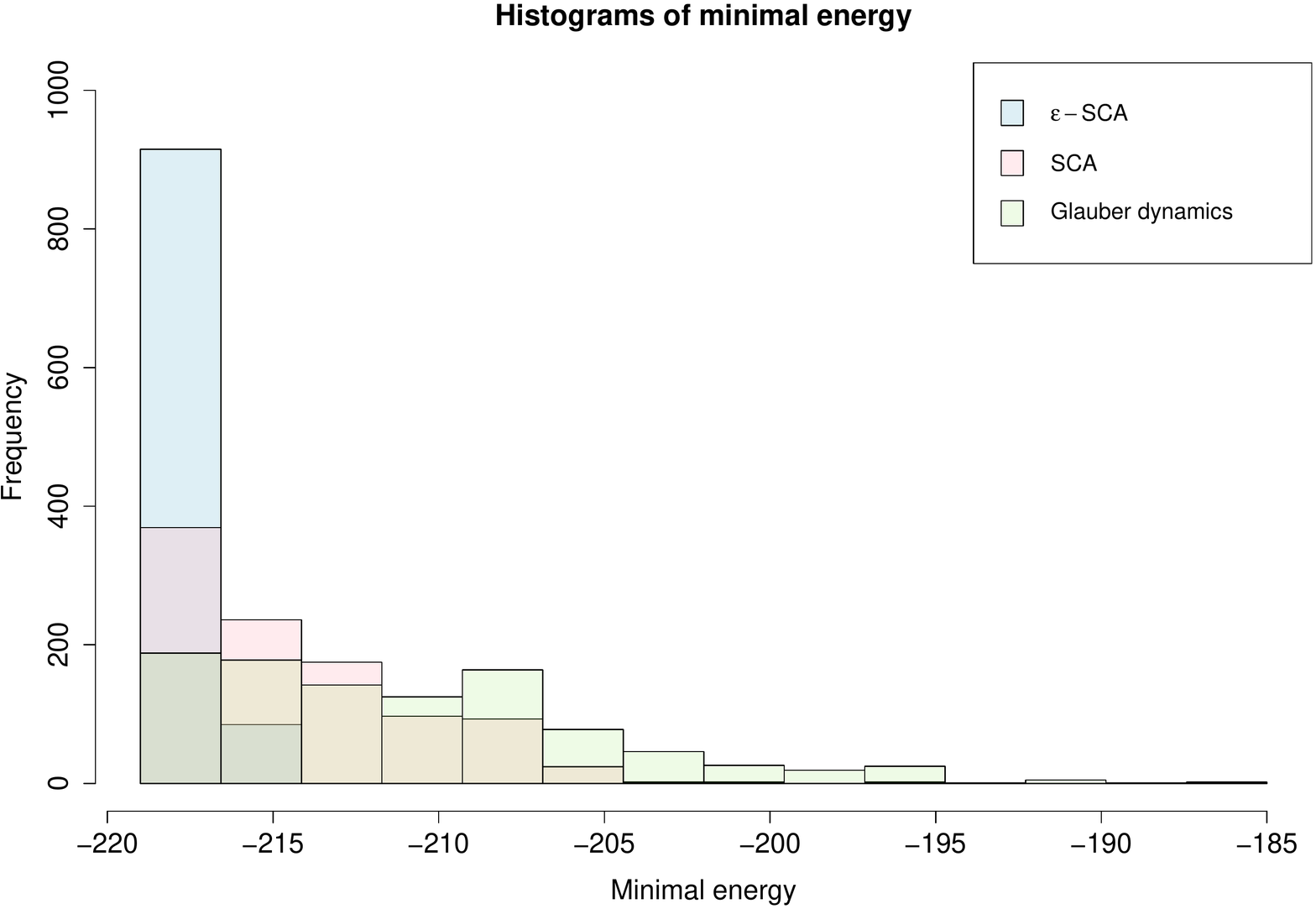}
                \caption{Max-cut problem with $p=0.1$, where $\vep = 0.6$}
                \label{fig:MC1}
        \end{subfigure}
        \begin{subfigure}[b]{0.45\textwidth}
                \includegraphics[width=\textwidth]{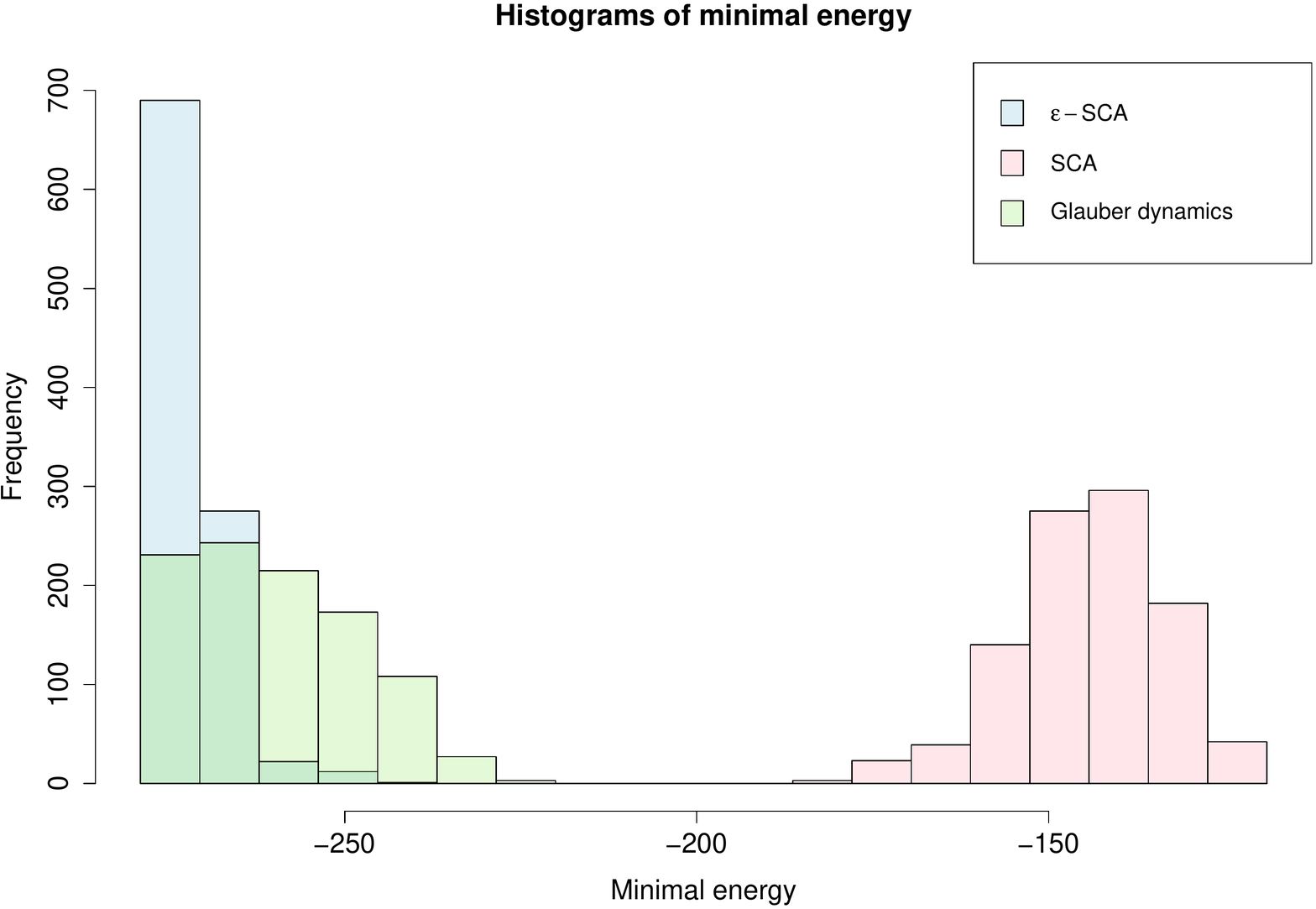}
                \caption{Max-cut problem with $p=0.9$, where $\vep = 0.1$}
                \label{fig:MC2}
        \end{subfigure}
                \begin{subfigure}[b]{0.45\textwidth}
                \includegraphics[width=\textwidth]{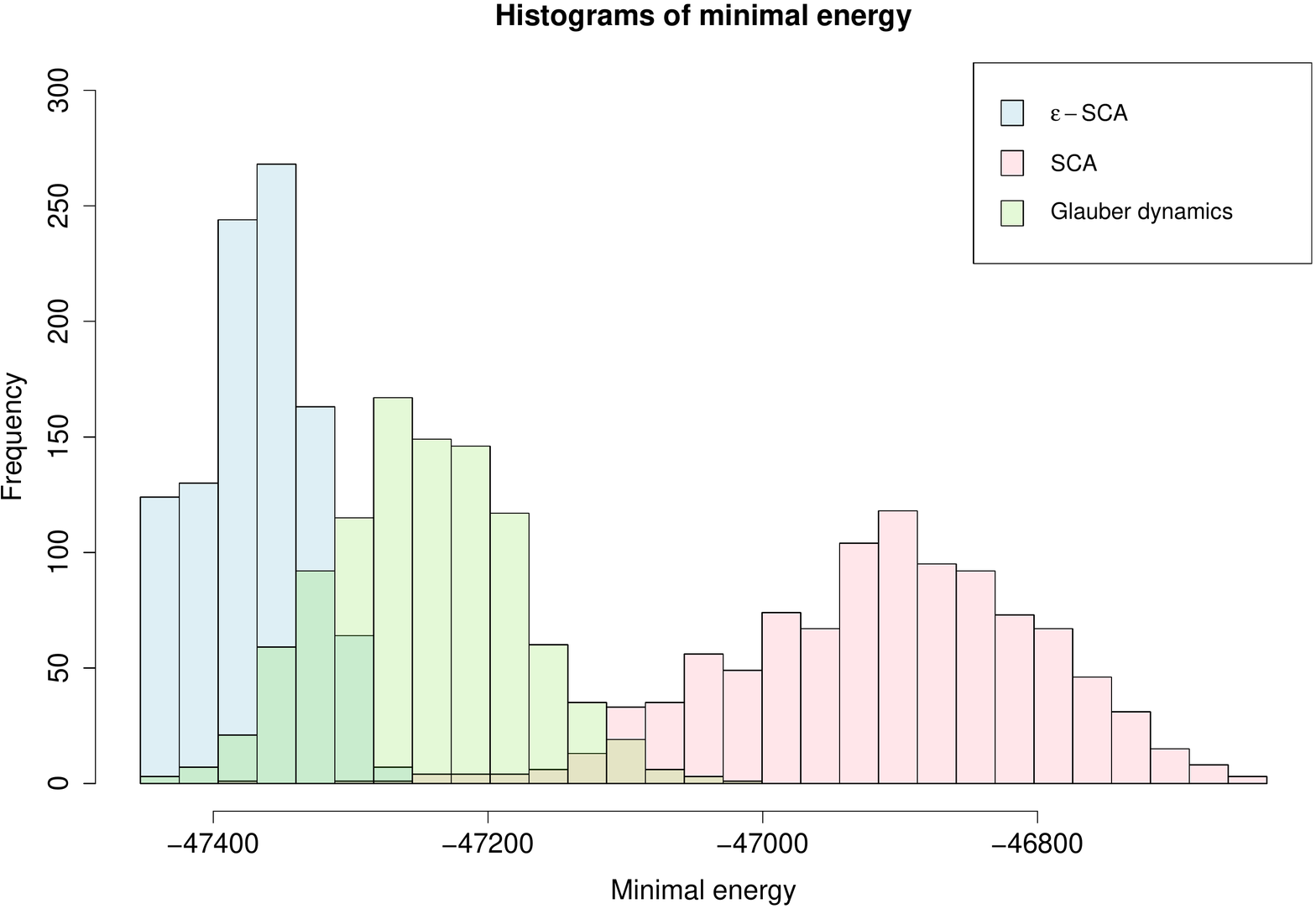}
                \caption{TSP, where $\vep = 0.3$}
                \label{fig:TSP1}
        \end{subfigure}
    \caption{Histograms obtained through the $\vep$-SCA, SCA, and Glauber dynamics applied to max-cut problems on Erd\" os-R\'enyi random graphs and the TSP, where $N = 100$.}
    \label{tab:hist2}
\end{figure}

\begin{table}
\caption{Success rate of obtaining a ground state associated to different Hamiltonians by using $\vep$-SCA, SCA, and Glauber dynamics}
\label{tab:SR}
\centering
\resizebox{\columnwidth}{!}{%
\begin{tabular}{|c|c|c|c|c|}
\hline
Models                                                                  & $\min(H)$ & $\vep$-SCA & SCA    & \begin{tabular}[c]{@{}c@{}}Glauber \\ dynamics\end{tabular} \\ \hline
\begin{tabular}[c]{@{}c@{}}Gaussian \\ spin glass\end{tabular}          & $-739.5749$ & $89.5\%$     & $31.5\%$ & $5.5\%$                                                       \\ \hline
\begin{tabular}[c]{@{}c@{}}Bernoulli\\ spin glass\\ ($p = 0.2$)\end{tabular} & $-600$  & $88.3\%$ & $0\%$    & $5.5\%$ \\ \hline
\begin{tabular}[c]{@{}c@{}}Bernoulli\\ spin glass\\ ($p = 0.5$)\end{tabular} & $-746$  & $42.9\%$ & $22.6\%$ & $3.6\%$ \\ \hline
\begin{tabular}[c]{@{}c@{}}Bernoulli\\ spin glass\\ ($p = 0.8$)\end{tabular} & $-2944$ & $100\%$  & $100\%$  & $100\%$ \\ \hline
\begin{tabular}[c]{@{}c@{}}Max-cut \\ problem \\ ($p = 0.1$)\end{tabular}   & $-219$  & $73.3\%$ & $11.6\%$ & $5\%$   \\ \hline
\begin{tabular}[c]{@{}c@{}}Max-cut \\ problem \\ ($p = 0.9$)\end{tabular} & $-279$      & $52.2\%$     & $0\%$      & $8.7\%$                                                       \\ \hline
TSP                                                                     &      -47453     &       $3.8\%$     & $0\%$    & $0.1\%$                                                         \\ \hline
\end{tabular}
}
\end{table}

\section{Discussion}

Through the inspection of Figs. \ref{tab:hist1} and \ref{tab:hist2} and the analysis of Table \ref{tab:SR}, we immediately observe that the simulated annealing algorithm based on $\vep$-SCA is the one which offered the highest success rates in obtaining ground states in all considered scenarios, being consistent with our previous findings in \cite{SCA21, SSS21}.
Since the study of Gaussian spin glasses and the max-cut problem on Erd\" os-R\'enyi random graphs was already addressed in \cite{SSS21}, let us solely concentrate on the analysis of Bernoulli spin glasses and the TSP.

In the minimization problem of Bernoulli spin glass Hamiltonians and in the TSP, we verified a phenomenon that is similar to one previously observed in the max-cut problem on Erd\" os-R\' enyi random graphs, where we noticed that the SCA struggles in obtaining low energy states when antiferromagnetic spin-spin interactions (i.e., when $J_{x,y} < 0$) are more prevalent than ferromagnetic interactions (i.e., when $J_{x,y} > 0$). In these cases, the pinning parameters, homogeneously taken as $q_x = \lambda / 2$, prevent the system from reaching certain low-energy configurations. Therefore, in order to overcome such a limitation, finding optimal values for the pinning parameters by refining the results from \cite{osky19} or increasing the simulation times by letting the temperature decrease at a slower rate can serve as alternatives for improving the accuracy of the SCA. In the cases where such a phenomenon described above was not present, the SCA reached ground states at a higher rate compared to the Glauber dynamics.

The $\vep$-SCA was applied to Bernoulli spin glasses with probabilities $p$ chosen as $p = 0.2, 0.5$ and $0.8$, considering the values for $\vep$ equal $\vep = 0.35, 0.75$ and $1$, respectively. For the case $p = 0.2$ (resp. $p = 0.8$), where antiferromagnetic (resp. ferromagnetic) interactions are dominant, the $\vep$-SCA obtained a higher success rate in finding a ground state of $88.3\%$ (resp. $100\%$). On the other hand, for the case $p = 0.5$, where the number of ferromagnetic and antiferromagnetic interactions are expected to be the same, the performance of the $\vep$-SCA was not as high as in the other two cases. For such a specific model, a ground state seems to be a spin configuration whose number of $+1$ and $-1$ spins are nearly the same, whereas there is a large number of configurations sharing that same feature whose energy is close to the energy of the ground state. So, for some reason related to the dispositions of the ferromagnetic and antiferromagnetic interactions, there is an increased chance for the dynamics of getting trapped in one of these configurations and not converging to the ground state. Again, considering slower decreasing temperature schedules may also improve the accuracy of the algorithm.

Although the success rate obtained for the $\vep$-SCA when applied to solve the TSP turned out to be smaller compared to the obtained when applied to the other problems, this value is still greater than those obtained when SCA and Glauber dynamics were applied instead. The $\vep$-SCA, with $\vep = 0.3$, and the Glauber dynamics obtained the same lowest energy configuration whose energy was equal $-47453$, corresponding to a trajectory with length equal $97$, with success rates of $3.8\%$ and $0.1\%$, respectively. On the other hand, the lowest energy configuration obtained for the SCA had energy equal to $-47369$, which corresponded to a trajectory with a length equal to $181$. Since the TSP can be expressed as the minimization problem of an Ising Hamiltonian with antiferromagnetic interactions and negative external fields, its energy landscape has been revealed to be much more complex compared to the other Hamiltonians. Such fact leads the algorithm to get stuck in configurations that are not ground states if we do not allow the temperature to drop at a sufficiently low speed. Moreover, the choice of the parameters $A$ and $B$ from its Hamiltonian may also affect the performance of the algorithms since a large value for the ratio $A/B$ may increase the chances of the dynamics stopping in a configuration that corresponds to a possible trajectory for the salesman but with non-minimal length, however, no general method for determining their optimal values is known. As an attempt to mitigate the difficulty of reaching a ground state at a higher rate without having to increase simulation times, we compared the performance of the $\vep$-SCA with the so-called Digital Annealer's Algorithm (often referred to as DA) \cite{A2019} applied to the TSP.  Such a comparison was performed under the same conditions as before, where the latter algorithm obtained the same ground state with a success rate of $8 \%$, see Fig. \ref{fig:TSP2}. In that case, the DA obtained not only a ground state at a higher rate, but also obtained low energy states more frequently. It follows that we found a particular problem where $\vep$-SCA's performance can be surpassed by a single spin-flip algorithm. Even though the DA offers superior performance, there is still no theoretical explanation for the reason why this algorithm performs better, specifically in this kind of problem. 

\begin{figure}
\centering
\includegraphics[width=0.45\textwidth]{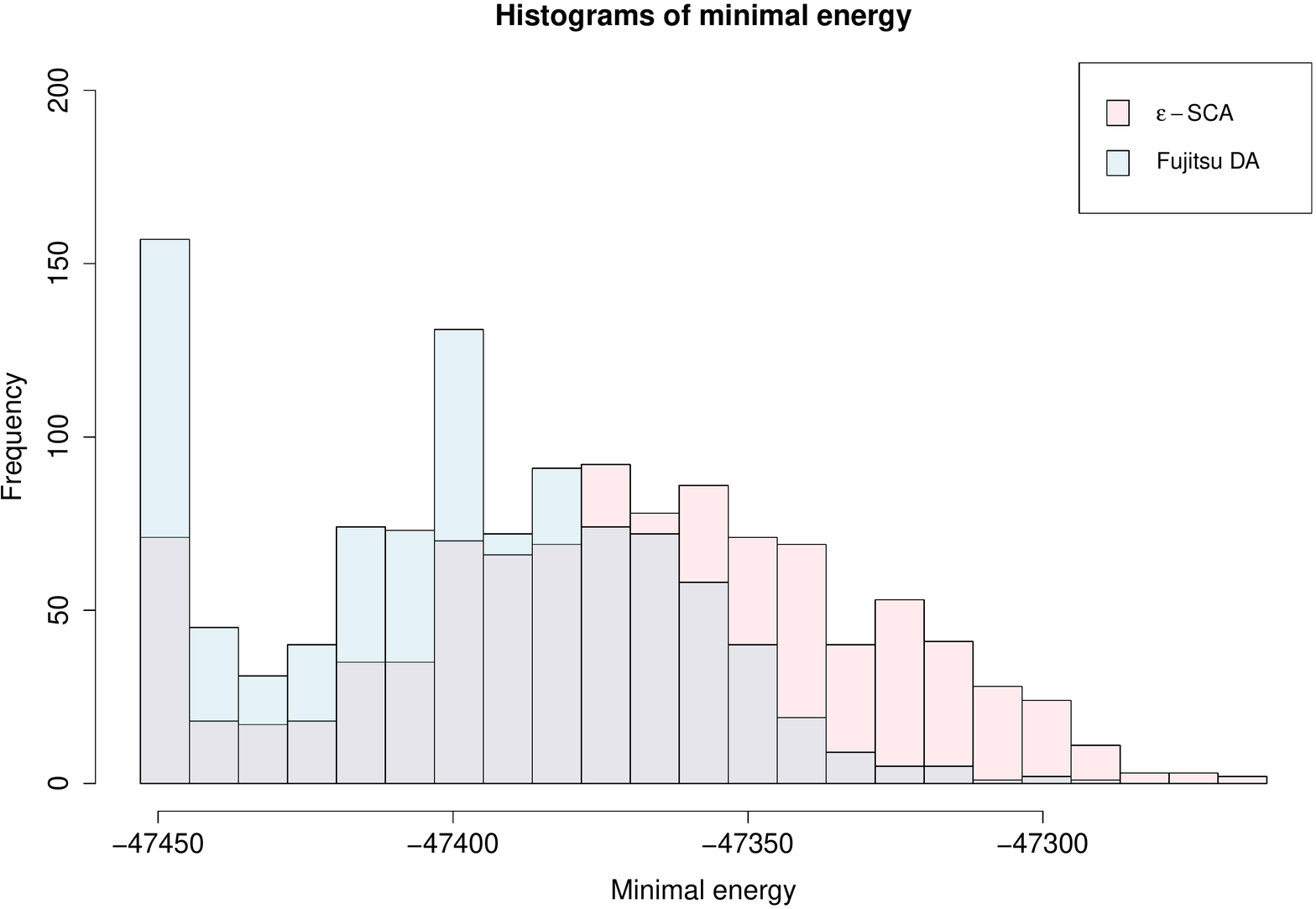}
\caption{Histograms corresponding to the application of the $\vep$-SCA (with $\vep = 0.3$) and the DA to the TSP.}\label{fig:TSP2}
\end{figure}

Similarly as in \cite{SSS21}, let us address the topic regarding the optimal value of $\vep$ for the $\vep$-SCA that should be chosen in order to maximize the success rate of reaching a ground state. 
The simulations were performed by considering the same Bernoulli spin glass Hamiltonians, with the same cooling schedule and simulation time as those used previously. In the plots from Fig. \ref{fig:SRMC}, each point that corresponds to a certain value of $\vep$  represents the success rate for the $\vep$-SCA in obtaining a ground state within $1000$ trials. Let us mention that in the case $p = 0.8$ the algorithm converged to the ground state with a success rate of $100\%$ for all tested values of $\vep$ in the interval $(0,1]$, for that reason, we did not illustrate the plot corresponding to this case. On the other hand, in Figs. \ref{fig:SRMC_1} and \ref{fig:SRMC_2}, we observe a growth tendency of such a rate as $\vep$ increases, but an  abrupt decrease appears as $\vep$ assumes values in a region close to $1$, which corresponds to the region where oscillation occurs. Note that for the Bernoulli spin glass with $p = 0.2$ such decrease starts occurring for smaller values of $\vep$ in comparison with the case $p = 0.5$ due to the fact that anti-ferromagnetic interactions dominate for smaller values of $p$ and the system becomes more susceptible to oscillate at low temperatures provided a larger number of spin-flips are allowed. Furthermore, as we consider larger values for $p$ so that ferromagnetic interactions play a more significant role in the Hamiltonian, the ground state tends to approach the configuration whose spins are all the same (either all $+1$ or all $-1$), and, in that case, the algorithm will converge to the ground states independently on the value of $\vep$ in the interval $(0,1]$.

\begin{figure}[ht]
        \centering   
        \begin{subfigure}[b]{0.45\textwidth}
                \includegraphics[width=\textwidth]{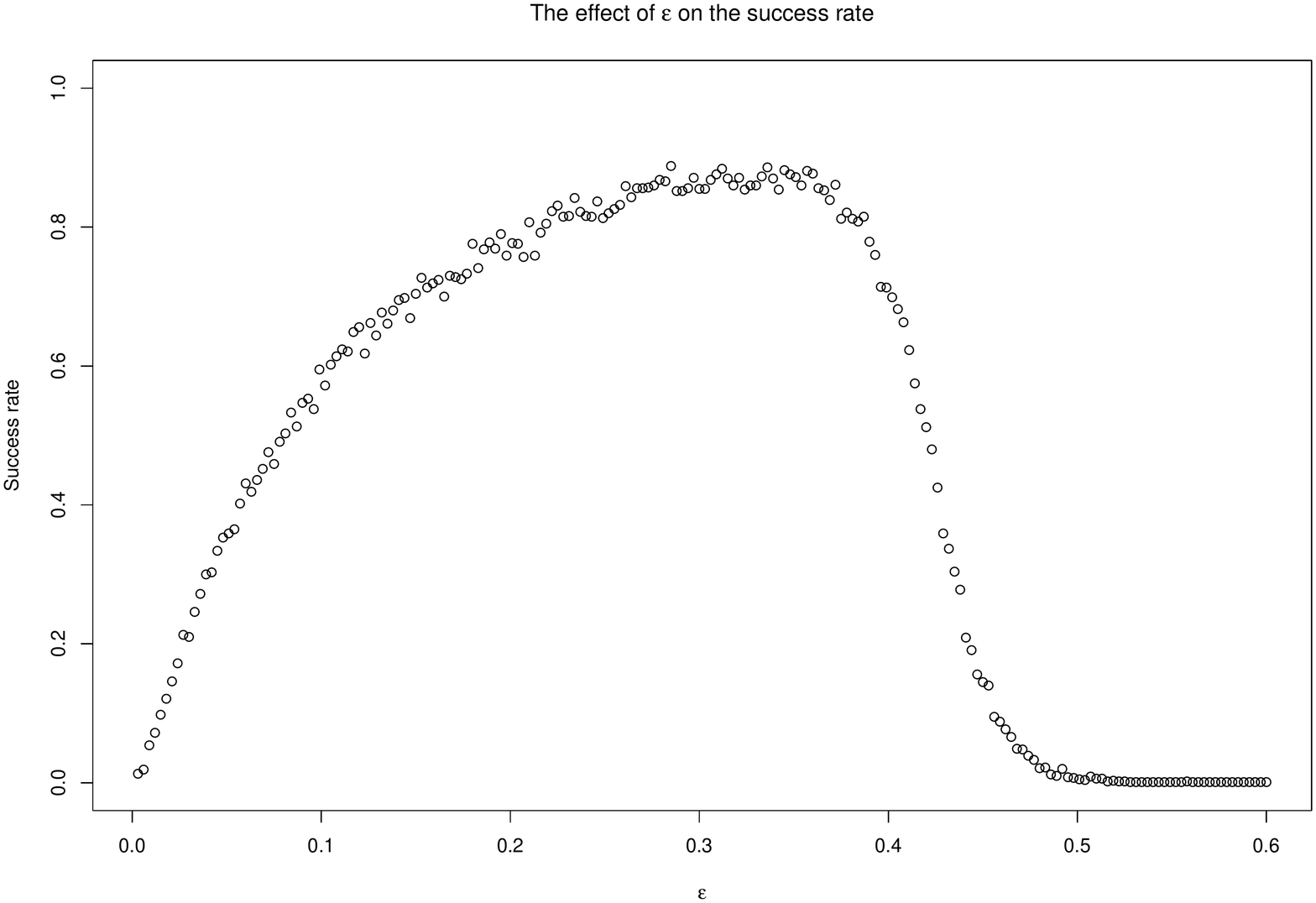}
                \caption{$N=100$ and $p = 0.2$}
                \label{fig:SRMC_1}
        \end{subfigure}
        \begin{subfigure}[b]{0.45\textwidth}
                \includegraphics[width=\textwidth]{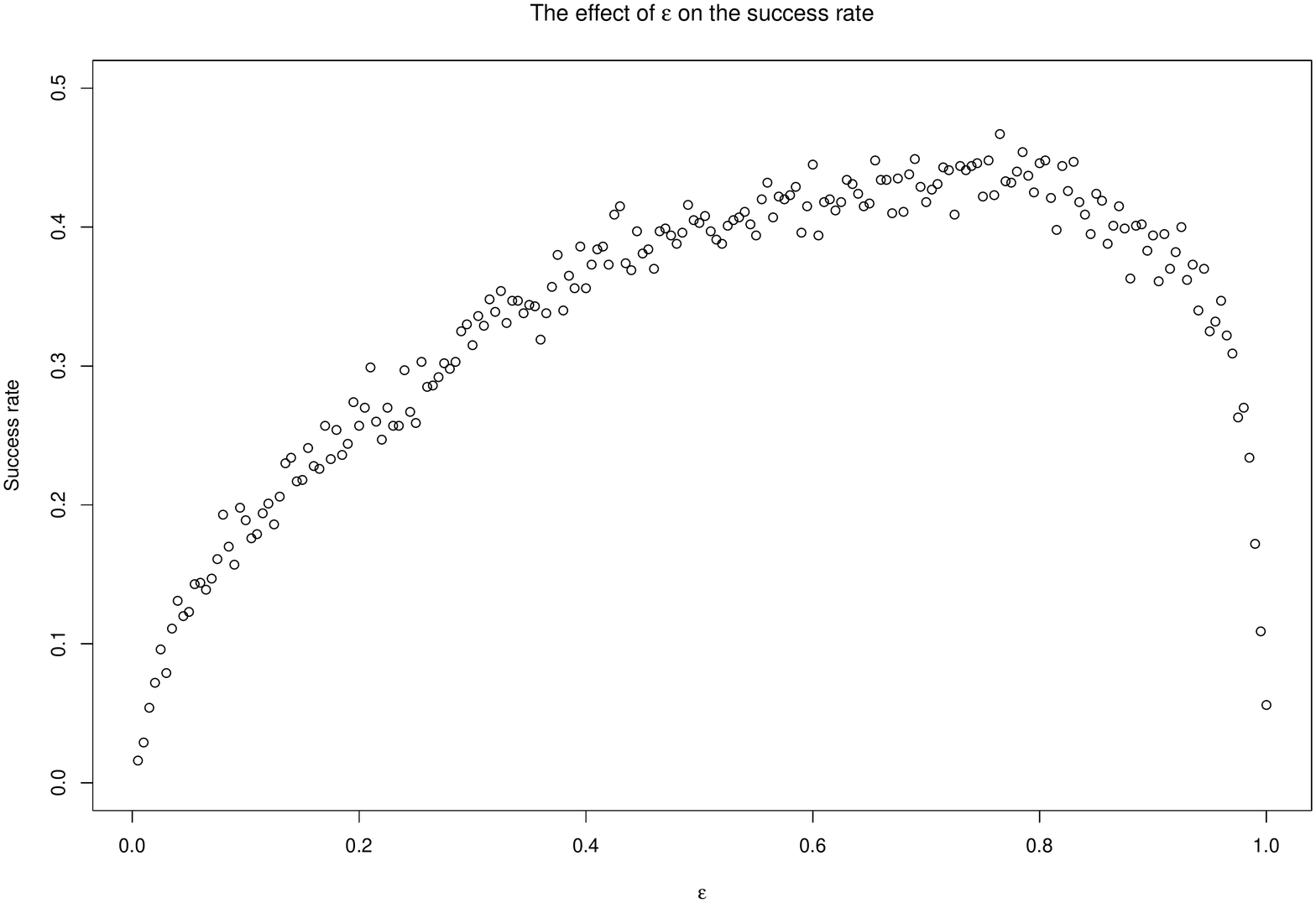}
                \caption{$N = 100$ and $p = 0.5$}
                \label{fig:SRMC_2}
        \end{subfigure}
    \caption{The dependence on $\vep$ of the success rate of finding a ground state for the Bernoulli spin glass Hamiltonian.}\label{fig:SRMC}
\end{figure}

\section{Conclusion}

In this work we extended the preliminary results showed in \cite{SSS21} by providing rigorous and practical results that can serve as a justification for the employment of simulated annealing algorithms based on parallel spin-flip updates for obtaining ground states. As a theoretical addition, it was possible to show that the mixing time for the $\vep$-SCA as well as for the SCA is at most proportional 
to $\log |V|$ provided the temperature is sufficiently high. 

The applications of simulated annealing for the $\vep$-SCA, SCA, and Glauber dynamics were illustrated by considering the minimization problems of Gaussian and Bernoulli spin glass Hamiltonians, the max-cut problem on  Erd\" os-R\' enyi random graphs, and the traveling salesman problem (TSP). The SCA outperformed the Glauber dynamics in some cases but failed in obtaining ground states for certain models where antiferromagnetic interactions play a significant role such as in max-cut problems on Erd\" os-R\' enyi random graphs with values for $p$ in an interval close to 1, Bernoulli spin glasses with $p$ in an interval close to $0$, and the TSP. However, the $\vep$-SCA exhibited the largest success rates in all scenarios, being consistent with our previous findings \cite{SCA21, SSS21}.

Even though the $\vep$-SCA performed significantly better than the other two algorithms in all cases, its hypothetical advantage of allowing a large number of spin flips and not suffering from the effect of pinning parameters did not guarantee a large success rate for the TSP. Due to its complexity, in order to improve the results for the TSP, larger simulation times with slowly decreasing temperatures are required. Moreover, for this specific class of problems, the so-called Digital Annealer's Algorithm (DA) showed to be more efficient than $\vep$-SCA. For that reason, theoretical and practical aspects concerning the properties of the DA are worth investigating.

The next step in future investigations consists of theoretically justifying the efficiency and limitations of the use of exponential cooling schedules for the SCA as well as for the $\vep$-SCA, aiming at providing some theoretical background that can be decisive in real-world applications. In order to gain more information that may allow us to classify the algorithms according to their appropriateness in solving certain classes of problems, more simulations considering a larger variety of problems including larger values of $N$ will be necessary. In order to do so, the employment of hardware accelerators such as GPUs are being considered, and a brief evaluation has been reported in \cite{SCA21}.

\end{document}